\documentclass[reqno,oneside,11pt]{amsart}
\usepackage{amsfonts,amssymb,amsthm,amsbsy,amscd,amsmath,bm,calc,circuitikz,color,dsfont,euscript, enumerate, geometry,graphicx,ifthen,latexsym,manfnt,marvosym,mathrsfs,marvosym,setspace,stmaryrd,textcomp,todonotes,tikz,tikz-cd, tcolorbox,upgreek,verbatim,xcolor}
\usepackage[utf8]{inputenc}
\usetikzlibrary{positioning,shapes,calc}
\definecolor{bulgarianrose}{rgb}{0.28, 0.02, 0.03}
\usepackage[linkcolor=bulgarianrose,citecolor=bulgarianrose,colorlinks=true,hypertexnames=false]{hyperref}
\allowdisplaybreaks

\newtheorem{theorem}{Theorem}[section]

\newtheorem{lemma}[theorem]{Lemma}

\theoremstyle{definition}

\newtheorem*{remark}{Remark}

\newtheorem*{example}{Example}
\newtheorem*{acknowledgement}{Acknowledgement}

\makeatletter
\def\imod#1{\allowbreak\mkern10mu({\operator@font mod}\,\,#1)}
\def\@textbottom{\vskip\z@\@plus 18pt}
\let\@texttop\relax
\makeatother

\title[Upper bound for the generalised Erd\H{o}s box problem]{An Analytic counting Framework for\\ the generalised Erd\H{o}s box problem}
\author{Subhankar Dash}
\author{Kaushik Majumder}

\address{\newline 
\newline (a) School of Mathematical Sciences\\ \newline National Institute of Science Education and Research (NISER) Bhubaneswar\\\newline Jatni, Khurda-$752050$, Odisha, India.
\newline (b) Homi Bhabha National Institute (HBNI)\\ \newline Training School Complex, Anushakti Nagar, Mumbai- $400094$, India.
\newline \textnormal{\textestimated-Mails}: (Subhankar Dash) {\tt subhankar.dash\MVAt niser.ac.in, isubha.apple\MVAt gmail.com}
\newline \textnormal{\textestimated-Mails}: (Kaushik Majumder) {\tt kaushikmajumder\MVAt niser.ac.in, kaushikbnmajumder\MVAt gmail.com}}


\begin{document}

\begin{abstract}
In this article, we develop an $r-$uniform analogue of the classical K\"{o}v\'{a}ri–S\'{o}s–Tur\'{a}n inequality. This yields an alternative proof of Erd\H{o}s's classical upper bound for complete $r$-partite $r$-uniform hypergraphs. More precisely, we establish that for finite non-empty sets $A_{1},\ldots,A_{r}$ with $|A_{1}|\leq\cdots\leq|A_{r}|$ and sufficiently large positive integer $n$,
 \[\mathrm{ex}(n,\mathbb{K}^{(r)}[A_{1},\ldots,A_{r}])=O\left(n^{r-\frac{1}{|A_{1}|\ldots|A_{r-1}|}}\right).\]
Our proof develops an analytic counting framework based on repeated applications of H\"{o}lder's inequality and the enumeration of configurations through multiple sums. The argument combines the principle of inclusion--exclusion with a discrete analogue of Fubini's theorem to obtain recursive estimates for extremal quantities. This provides a unified analytic perspective on the generalized Erd\H{o}s box problem.
\end{abstract}
\maketitle

\section{Introduction}
A \emph{hypergraph} $H$ is a pair $(V,E)$, where $V$ is a non-empty set, called the \emph{vertex set} and $E\subset 2^{V}$ is called the \emph{edge set}. In addition, it is called an \emph{$r-$uniform hypergraph} or simply an \emph{$r-$graph} if $E\subset\binom{V}{r}$, where $r\geq2$ is an integer. For an $r-$graph  $H$, the \emph{Tur\'{a}n number} $\mathrm{ex}(n, H)$ is the maximum number of edges in an $r-$graph on $n$ vertices, which contains no copies of $H$. In other words,
\begin{equation*}
\mathrm{ex}(n,H)=\max\bigg\{|E(G)|:\text{the $r-$graph $G$ on $n-$vertices is }H-\text{free}\bigg\}.
\end{equation*}
Many authors denote this quantity by $\mathrm{ex}^{(r)}(n,H)$ or $\mathrm{ex}_{r}(n,H)$. In notation, $\mathrm{ex}(n,H)$, the $r-$graph $H$ is called the \emph{forbidden $r-$graph}.  

An $r-$graph is called $k-$partite, where the integer $k\geq r$, if its vertex set admits a partition into $k$ many non-empty parts (say) $V_{1}\sqcup\ldots\sqcup V_{k}$ and each edge contains at most one vertex from each part. We denote by $\mathbb{K}^{(r)}[A_{1},\ldots,A_{r}]$, the \emph{complete $r-$partite $r-$graph} with vertex set partition $A_{1}\sqcup\ldots\sqcup A_{r}$. Its edge set consists of all set of size $r$, of the form $\{a_{1},\ldots,a_{r}\}$, where $a_{i}\in A_{i}$ for each $i\in[r]$. Throughout this article, the forbidden  $r-$graph is taken to be $\mathbb{K}^{(r)}[A_{1},\ldots,A_{r}]$, where $A_{1},\ldots,A_{r}$ are finite non-empty mutually disjoint sets with $|A_1|\leq\cdots\leq|A_{r}|$. 

After proving the celebrated Erd\H{o}s--Stone theorem, Paul Erd\H{o}s and his distinguished collaborators considered the analogous problem in the context of $r-$graphs. As a follow-up, the following result was established:
\begin{theorem}[Erd\H{o}s]\cite{MR183654}\label{Erdos Result}
For finite non-empty sets $A_{1},\ldots,A_{r}$ with $|A_{1}|\leq\ldots\leq|A_{r}|$ and for sufficiently large positive integers $n$,
\begin{equation*}
\Omega\left(n^{r-\frac{C}{|A_{1}|\ldots|A_{r-1}|}}\right)=\mathrm{ex}(n,\mathbb{K}^{(r)}[A_{1},\ldots,A_{r}])=O\left(n^{r-\frac{1}{|A_{1}|\ldots|A_{r-1}|}}\right),
\end{equation*}
where $C>1$ is a constant independent of $\{n,r,|A_{1}|,\ldots,|A_{r}|\}$.
\end{theorem}
The lower bound was proved using the first moment method, while the upper bound was obtained through induction combined with a combinatorial neighbourhood intersection lemma. Mubayi \cite{MR1909504} employed the union method to give a second proof of the lower bound.

The graph version (that is $r=2$) of this problem, that is  determining $\mathrm{ex}(n,\mathbb{K}^{(2)}[A_{1},A_{2}])$ is famously known as \emph{Zarankiewicz problem}. As explained in the textbook \cite[Chapter~VI, \S2]{MR506522}, following the argument of K\"{o}vari, S\'{o}s and Tur\'{a}n, one obtains the following inequality for each $\mathbb{K}^{(2)}[A_{1},A_{2}]-$free graph $G$ on $n$ vertices:
\begin{equation*}
\underset{x\in V(G)}{\sum}\binom{\deg_{G}(x)}{|A_{1}|}\leq(|A_{2}|-1)\binom{n}{|A_{1}|}.    
\end{equation*}
Using this inequality, it can be concluded that for each positive integer $n$, 
\begin{equation*}
 \mathrm{ex}(n,\mathbb{K}^{(2)}[A_{1},A_{2}])\leq\frac{1}{2}\left(|A_{2}|-1\right)^{\frac{1}{|A_{1}|}}n^{2-\frac{1}{|A_{1}|}}+\frac{|A_{1}|-1}{2}n.   
\end{equation*}
The main contribution of this article is an $r-$graph analogue of the classical K\"{o}v\'{a}ri–S\'{o}s–Tur\'{a}n counting inequality. The proof follows a different approach from Erd\H{o}s's original argument. The graph case is established by the K\"{o}v\'{a}ri–S\'{o}s–Tur\'{a}n theorem \cite{MR65617}. Erd\H{o}s's proof of the $r-$graph case proceeds differently. It relies on an intersection lemma together with induction on the uniformity. Our proof approach is to turn the classical K\"{o}v\'{a}ri–S\'{o}s–Tur\'{a}n counting inequality into inequality~\eqref{inequality r-graph}. It develops recursive analytic estimates. It is based on indicator-function representations, repeated applications of H\"{o}lder's inequality, a discrete analogue of Fubini's theorem and Bonferroni's inequality. We adopt the indicator-function formulation advocated by Gowers \cite{MR2195580}. This enables us to translate the combinatorial counting problem into an analytic one. In particular, subhypergraph counts are expressed as multilinear sums of indicator functions, which can then be estimated analytically. Repeated $\mathsf{L}^{p}-$estimates are the central feature of our argument. We refer to this analytic counting framework as the \emph{$\mathsf{L}^{p}-$method}. As an application of this method, we prove a new counting inequality to prove the upper bound in Theorem~\ref{Erdos Result}. More precisely, we establish the following theorem.

\begin{theorem}\label{upper bound of Erdos Result}
 For non-empty finite sets $A_{1},\ldots,A_{r}$ with  $|A_{1}|\leq\ldots\leq|A_{r-1}|\leq|A_{r}|$, and for each positive integer $n$, 
\begin{equation*}
\mathrm{ex}(n,\mathbb{K}^{(r)}[A_{1},\ldots,A_{r}])\leq n^{r-\frac{1}{|A_{1}|\ldots|A_{r-1}|}}\frac{1}{r}\Bigg(\frac{|A_{r}|-1}{|A_{1}|!\ldots|A_{r-1}|!}+\frac{r(r-1)}{r!}\binom{|A_{r-1}|}{2}\Bigg)^{\frac{1}{|A_1|\ldots|A_{r-1}|}}.    
\end{equation*}   
\end{theorem}
That is, for sufficiently large positive integers $n$, we have 
\begin{equation*}
\mathrm{ex}(n,\mathbb{K}^{(r)}[A_{1},\ldots,A_{r}])=O\left(n^{r-\frac{1}{|A_{1}|\ldots|A_{r-1}|}}\right)
\end{equation*}
and the constant term involved in the above expression depends on  $\{r,|A_{1}|,\ldots,|A_{r-1}|,|A_{r}|\}$. 

In order to extend this problem, Mubayi conjectured \cite{MR1909504} that the lower bound can be improved to $C=1$, which also means that the right hand side is conjectured to give the correct order of magnitude. In other words, the conjecture says that for sufficiently large positive integers $n$,
\begin{equation*}
\mathrm{ex}(n,\mathbb{K}^{(r)}[A_{1},\ldots,A_{r}])=\Theta\left(n^{r-\frac{1}{|A_{1}|\cdots|A_{r-1}|}}\right).
\end{equation*}
The conjecture is well known to hold in a number of special cases. A proof of the partially constructive lower bound, based on the Combinatorial Nullstellensatz, was obtained for the special case $|A_{1}|=2=\cdots=|A_{r}|$ \cite{MR4735161}. This conjecture is claimed to be solved in \cite{pohoata2021norm}.  Note here that when $|A_{1}|=2=\cdots=|A_{r}|$, the problem of determining asymptotic behaviour of $\mathrm{ex}(n,\mathbb{K}^{(r)}[A_{1},\ldots,A_{r}])$ is commonly known as \emph{Erd\H{o}s box problem}. For each integer $r\geq2$, there exists a pair of positive integers $(s,t)$ satisfies $r(s-1)<(2^{r}-1)t$, (i.e. both integers $s$ and $t$ depend on $r$) for sufficiently large positive integers $n$, 
\begin{equation*}
\mathrm{ex}(n,\mathbb{K}^{(r)}[A_{1},\ldots,A_{r}])=\Omega(n^{r-\frac{t}{s}}),
\end{equation*}
where $|A_{1}|=2=\cdots=|A_{r}|$. This is the best known lower bound for Erd\H{o}s box problem \cite{MR4328787}. 
When at least one of $|A_{1}|,\ldots,|A_{r}|$ differs from $2$, we refer to the corresponding Tur\'{a}n problem as  \emph{generalised Erd\H{o}s box problem}. 
It concerns the determination of the Tur\'{a}n number of $r-$uniform hypergraphs that are free of the complete $r-$partite $r-$uniform hypergraph, whose partite sets are $A_{1},\ldots,A_{r}$ and at least one of the partite sets have size different from $2$.

\section{Prerequisite Results}
We use the following version of the Principle of inclusion and exclusion. 
\begin{theorem}[Principle of inclusion and exclusion-general version]\label{General PIE}
Let $\Omega_{1},\ldots,\Omega_{n}$ be a collection of non-empty subsets of $\Omega$. For each function $f:\underset{i\in[n]}{\bigcup}\Omega_{i}\to\mathbb{R}$, the sum of $f(x)$ over the union of these sets is given exactly by:
\begin{equation*}
\underset{x\in\underset{i\in[n]}{\bigcup}\Omega_{i}}{\sum}f(x)=\underset{\emptyset\neq S\subseteq[n]}{\sum}(-1)^{|S|-1} \underset{x\in\underset{i\in S}{\bigcap}\Omega_{i}}{\sum}f(x).
\end{equation*}
\end{theorem}

Taking $f(x)=\mathds{1}_{\underset{i\in[n]}{\bigcup}\Omega_{i}}(x)$ for each $x\in\underset{i\in[n]}{\bigcup}\Omega_{i}$ deduces the usual principle of inclusion and exclusion. The following is an alternating inequality obtained from the Theorem~\ref{General PIE}.

\begin{theorem}[Bonferroni's Inequality]\label{Bonferroni's Inequality}
For each $i\in[n]$, let $\Omega_{i}$ be a non-empty set. Then for each function $f:\underset{i\in[n]}{\bigcup}\Omega_{i}\to[0,\infty)$, one the following estimate holds:-
\begin{equation*}
 \underset{x\in\underset{i\in[n]}{\bigcup}\Omega_{i}}{\sum}f(x)\begin{cases}
 \leq\underset{j=1}{\overset{k}{\sum}}(-1)^{j-1}\underset{S\in\binom{[n]}{j}}{\sum}\left(\underset{x\in\underset{s\in S}{\cap}\Omega_{s}}{\sum}f(x)\right)&\textit{ for each odd integer }k\in[1,n],\\
 \geq\underset{j=1}{\overset{k}{\sum}}(-1)^{j-1}\underset{S\in\binom{[n]}{j}}{\sum}\left(\underset{x\in\underset{s\in S}{\cap}\Omega_{s}}{\sum}f(x)\right)&\textit{ for each even integer }k\in[1,n].
 \end{cases}   
\end{equation*}
\end{theorem}
As the proof for this version was not tracked, we reconstruct the proof here.
\begin{proof}
We prove this using induction on $n$. For each non-empty sets $\Omega_{1}$ and $\Omega_{2}$ and each function $f:\Omega_{1}\cup\Omega_{2}\to[0,\infty)$,
\begin{equation*}
\underset{x\in\Omega_{1}\cup\Omega_{2}}{\sum}f(x)=\underset{x\in\Omega_{1}}{\sum}f(x)+\underset{x\in\Omega_{2}}{\sum}f(x)-\underset{x\in\Omega_{1}\cap\Omega_{2}}{\sum}f(x)
\leq\underset{x\in\Omega_{1}}{\sum}f(x)+\underset{x\in\Omega_{2}}{\sum}f(x).    
\end{equation*}
This shows both the inequalities hold for the base case $n=2$. 

\noindent\textit{Hypothesis :} Suppose the inequalities hold for each collection of $n-1$ many non-empty sets and each non-negative function defined on the union of these sets. 

\noindent\textit{Inductive Step :} Let $\Omega_{1},\ldots,\Omega_{n}$ be $n-$non-empty sets and $f:\underset{i\in[n]}{\bigcup}\Omega_{i}\to[0,\infty)$. We construct $\Lambda=\underset{i\in[n-1]}{\bigcup}\Omega_{i}$. Then
\begin{equation*}
\underset{x\in\underset{i\in[n]}{\bigcup}\Omega_{i}}{\sum}f(x)=\underset{x\in\Lambda}{\sum}f(x)+\underset{x\in\Omega_{n}}{\sum}f(x)-\underset{x\in\Lambda\cap\Omega_{n}}{\sum}f(x).   
\end{equation*}
Let $k\in[n-1]$ be an odd integer (for case $k=n$, we settle the equality using Theorem~\ref{General PIE}). We establish the upper bound in the statement.  Applying the induction hypothesis for $\Omega_{1},\ldots,\Omega_{n-1}$,
\begin{align*}
\underset{x\in\Lambda}{\sum}f(x)\le\underset{j=1}{\overset{k}{\sum}}(-1)^{j-1}\underset{S_j\in\binom{[n-1]}{j}}{\sum}\underset{x\in\underset{s\in S_j}{\bigcap}\Omega_{s}}{\sum}f(x).
\end{align*} 

Again here $k-1$ equals an even integer, thus applying the induction hypothesis for at most $n-1$ many non-empty subsets $(\Omega_{1}\cap\Omega_{n}),\ldots,(\Omega_{n-1}\cap\Omega_{n})$,
\begin{align*}
\underset{x\in\Lambda\cap\Omega_{n}}{\sum}f(x)&\geq\underset{j=1}{\overset{k-1}{\sum}}(-1)^{j-1}\underset{S_j\in\binom{[n-1]}{j}}{\sum}\underset{x\in\underset{s\in S_j}{\bigcap}(\Omega_{s}\cap\Omega_{n})}{\sum}f(x)\\
&=\underset{j'=2}{\overset{k}{\sum}}(-1)^{j'-2}\underset{S_{j'-1}\in\binom{[n-1]}{j'-1}}{\sum}\underset{x\in\left(\underset{s\in S_{j'-1}}{\bigcap}\Omega_{s}\right)\cap\Omega_{n}}{\sum}f(x).
\end{align*}

Hence, using the above two estimates, we have 
\begin{align*}
\underset{x\in\underset{i\in[n]}{\bigcup}\Omega_{i}}{\sum}f(x)&\le\underset{j=1}{\overset{k}{\sum}}(-1)^{j-1}\underset{S_j\in\binom{[n-1]}{j}}{\sum}\underset{x\in\underset{s\in S_j}{\bigcap}\Omega_{s}}{\sum}f(x)+\underset{x\in\Omega_{n}}{\sum}f(x)\\
&\hspace{7em}-\underset{j'=2}{\overset{k}{\sum}}(-1)^{j'-2}\underset{S_{j'-1}\in\binom{[n-1]}{j'-1}}{\sum}\underset{x\in\left(\underset{s\in S_{j'-1}}{\bigcap}\Omega_{s}\right)\cap\Omega_{n}}{\sum}f(x)\\
&=\underset{s\in[n-1]}{\sum}\underset{x\in\Omega_s}{\sum}f(x)+\underset{j=2}{\overset{k}{\sum}}(-1)^{j-1}\underset{S_j\in\binom{[n-1]}{j}}{\sum}\underset{x\in\underset{s\in S_j}{\bigcap}\Omega_{s}}{\sum}f(x)+\underset{x\in\Omega_{n}}{\sum}f(x)\\
&\hspace{7em}+\underset{j'=2}{\overset{k}{\sum}}(-1)^{j'-1}\underset{S_{j'-1}\in\binom{[n-1]}{j'-1}}{\sum}\underset{x\in\left(\underset{s\in S_{j'-1}}{\bigcap}\Omega_{s}\right)\cap\Omega_{n}}{\sum}f(x)\\
&=\underset{s\in[n]}{\sum}\underset{x\in\Omega_s}{\sum}f(x)+\\
&\hspace{1cm}+\underset{j=2}{\overset{k}{\sum}}(-1)^{j-1}\left[\underset{S_j\in\binom{[n-1]}{j}}{\sum}\underset{x\in\underset{s\in S_j}{\bigcap}\Omega_{s}}{\sum}f(x)+\underset{S_{j-1}\in\binom{[n-1]}{j-1}}{\sum}\underset{x\in\left(\underset{s\in S_{j-1}}{\bigcap}\Omega_{s}\right)\cap\Omega_{n}}{\sum}f(x)\right]\\
&=\underset{s\in[n]}{\sum}\underset{x\in\Omega_s}{\sum}f(x)+\underset{j=2}{\overset{k}{\sum}}(-1)^{j-1}\underset{S_j\in\binom{[n]}{j}}{\sum}\underset{x\in\underset{s\in S_j}{\bigcap}\Omega_{s}}{\sum}f(x)\\
&=\underset{j=1}{\overset{k}{\sum}}(-1)^{j-1}\underset{S_j\in\binom{[n]}{j}}{\sum}\underset{x\in\underset{s\in S_j}{\bigcap}\Omega_{s}}{\sum}f(x).
\end{align*}    
This establishes the result for each odd integer $k\in[1,n]$. Using a similar argument, we have that for each even integer $k'\in[1,n]$, the lower bound in the statement holds. 
\end{proof}

\begin{remark}
For each $i\in[n]$, let $\Omega_{i}$ be a non-empty subset of $\Omega$ and $f:\Omega\to[0,\infty)$. Then
\begin{align*}
\underset{x\in\Omega\smallsetminus\left(\underset{i\in[n]}{\bigcup}\Omega_{i}\right)}{\sum}f(x)&=\underset{x\in\Omega}{\sum}f(x)-\underset{x\in\underset{i\in[n]}{\bigcup}\Omega_{i}}{\sum}f(x).
\end{align*}
We note that the negative sign in the above equality reverses the Bonferroni bounds: for odd positive integers $k$, it yields a lower bound, while for even positive integers $k$, it yields an upper bound.
\end{remark}

We use the following discrete version of the H\"{o}lder's inequality.
\begin{theorem}[H\"{o}lder's Inequality]\label{Holder's Inequality}
Let $p\in[1,\infty)$ and $a_{i}\geq0$ for each $i\in A$, where $A$ denotes a non-empty index set, then
\begin{equation*}
\frac{1}{|A|^{p-1}}\bigg(\underset{i\in A}{\sum}a_{i}\bigg)^{p}\leq\underset{i\in A}{\sum}a_{i}^{p}.
\end{equation*}
\end{theorem}

We also require the following lemma in the proof of the upper bound of Theorem~\ref{Erdos Result}.

\begin{lemma}\label{exponent}
Let $a_{1},\ldots,a_{k}$ be real numbers. If  for each $i\in\{2,\ldots,k\}$ 
\begin{align*}
R_{i}&=k+1-i+(a_{1}+\ldots+a_{i-1}),\\
P_{i}&=a_{i}\ldots a_{k},
\end{align*}
$R_{1}=k$,$P_{1}=a_{1}\ldots a_{k}$, $R_{k+1}=a_{1}+\ldots+a_{k}$ and $P_{k+1}=1$, then 
\begin{align*}
\underset{i\in[k]}{\sum}(k+1-i+a_{1}+\ldots+a_{i-1})(a_{i}-1)a_{i+1}\ldots a_{k}&=\underset{i\in[k]}{\sum}R_{i}(P_{i}-P_{i+1})\\
&=(k+1)a_{1}\ldots a_{k}-(a_{1}+\ldots+a_{k})-1.   
\end{align*}
\end{lemma}

\begin{proof}
 For each $i\in[k]$, we have $R_{i+1}=R_{i}+a_{i}-1$ leads, the required summation equals $R_{1}P_{1}-R_{k+1}P_{k+1}+P_{1}-P_{k+1}=(k+1)a_{1}\ldots a_{k}-(a_{1}+\ldots+a_{k})-1$.    
\end{proof}

\section{Proof of the Theorem~\ref{upper bound of Erdos Result}}

Let $H$ be a $\mathbb{K}^{(r)}[A_{1},\ldots,A_{r}]-$free $r-$graph on $n-$vertices where each $A_{i}$ is a non-empty finite set. For subsets $S_{1},\ldots,S_{r-1}$ of $V(H)$, we construct
\begin{equation*}
 \mathds{1}_{E(H)}(x,S_{1},\ldots,S_{r-1})=\underset{\substack{s_{1}\in S_{1}\\\vdots\\s_{r-1}\in S_{r-1}}}{\prod}\mathds{1}_{E(H)}(\{x,s_{1},\ldots,s_{r-1}\}).   
\end{equation*}
We observe that if $s\in S_{i}\cap S_{j}$ for some distinct $i,j\in[r-1]$, then $\mathds{1}_{E(H)}(x,S_{1},\ldots,S_{r-1})=0$. Thus, the above product is non-zero only when the subsets $S_{1},\ldots,S_{r-1}\subseteq V(H)$ are pairwise disjoint.

Finally, we construct
\begin{equation*}
\mathfrak{S}=\left\{(x,S_{1},\ldots,S_{r-1}\}):S_{i}\in\binom{V(H)}{|A_{i}|},\ \forall i\in\left[r-1\right]\textnormal{ and }
    \mathds{1}_{E(H)}(x,S_{1},\ldots,S_{r-1})=1\right\}.
\end{equation*}
Applying Fubini's theorem (i.e., interchanging the order of summation), we count $\mathfrak{S}$ in two equivalent ways:
\begin{align*}
|\mathfrak{S}|&=\underset{\substack{S_{1},\ldots,S_{r-1}\subseteq V(H)\\|s_{1}|=|A_{1}|,\ldots,|S_{r-1}|=|A_{r-1}|}}{\sum}\underset{x\in V(H)}{\sum} \mathds{1}_{E(H)}(x,S_{1},\ldots,S_{r-1})\hspace{5em}(\textnormal{First way})\\
&=\underset{x\in V(H)}{\sum} \underset{\substack{S_{1},\ldots,S_{r-1}\subseteq V(H)\\|s_{1}|=|A_{1}|,\ldots,|S_{r-1}|=|A_{r-1}|}}{\sum}\mathds{1}_{E(H)}(x,S_{1},\ldots,S_{r-1})\hspace{5em}(\textnormal{Second way}).
\end{align*}
We estimate the first expression from above and the second from below.

\noindent\textbf{First Way:} Let $S_{1},\ldots,S_{r-1}$ be subsets of $V(H)$ and $|S_{i}|=|A_{i}|$ for each $i\in[r-1]$.

\noindent\textit{Claim.}
\begin{equation*}
\underset{x\in V(H)}{\sum} \mathds{1}_{E(H)}(x,S_{1},\ldots,S_{r-1})\leq|A_{r}|-1.
\end{equation*}
\begin{proof}[\tt{Proof of claim} :]\renewcommand{\qedsymbol}{}
If the subsets $S_{1},\ldots,S_{r-1}$ of $V(H)$ are not pairwise disjoint, then the sum is $0$ and the result holds trivially. Hence we assume that they are pairwise disjoint. 

If possible suppose 
\begin{equation*}
\underset{x\in V(H)}{\sum} \mathds{1}_{E(H)}(x,S_{1},\ldots,S_{r-1})\geq|A_{r}|.    
\end{equation*}
Then we can choose $x_{1},\ldots,x_{|A_{r}|}\in V(H)$, with $\mathds{1}_{E(H)}(x_k,S_{1},\ldots,S_{r-1})=1$ for each $k\in[|A_{r}|]$.
This implies for $U=\{x_{1},\ldots,x_{|A_{r}|}\}$ the induced $r-$subgraph $H[U,S_{1},\ldots,S_{r-1}]$ of the $r-$graph $H$, contains a copy of $\mathbb{K}^{(r)}(A_{1},\ldots,A_{r})$. A contradiction arises, which establishes the claim.
\end{proof}

There are $\frac{|(S_{1}\sqcup\ldots\sqcup S_{r-1})|)!}{|A_{1}|!\ldots|A_{r-1}|!}$ many partitions of $S_{1}\sqcup\ldots\sqcup S_{r-1}$ of the form $S'_{1}\sqcup\ldots\sqcup S'_{r-1}$ where $|S'_{1}|=|A_{1}|,\ldots,|S'_{r-1}|=|A_{r-1}|$, i.e.
\begin{align*}
&\bigg|\bigg\{(S'_{1},\ldots,S'_{r-1}):\underset{i\in[r-1]}{\sqcup}S'_{i}=\underset{i\in[r-1]}{\sqcup}S_{i},|S'_{i}|=|A_{i}|\textup{ and }S'_{i}\subset\underset{i\in[r-1]}{\sqcup}S_{i},\ \forall i\in[r-1]\bigg\}\bigg|\\
&\hspace{3cm}=\frac{|\underset{i\in[r-1]}{\sqcup}S_{i}|!}{\underset{i\in[r-1]}{\prod}|A_{i}|!}=\frac{\left(\underset{i\in[r-1]}{\sum}|A_{i}|\right)!}{\underset{i\in[r-1]}{\prod}|A_{i}|!}.  
\end{align*}
Using the previous claim, we have 
\begin{align*}
|\mathfrak{S}|=&\underset{\substack{S_{1},\ldots,S_{r-1}\subseteq V(H)\\|s_{1}|=|A_{1}|,\ldots,|S_{r-1}|=|A_{r-1}|}}{\sum}\underset{x\in V(H)}{\sum} \mathds{1}_{E(H)}(x,S_{1},\ldots,S_{r-1})\\
\leq&\underset{\substack{S_{1},\ldots,S_{r-1}\subseteq V(H)\\|s_{1}|=|A_{1}|,\ldots,|S_{r-1}|=|A_{r-1}|}}{\sum}(|A_{r}|-1)\\
&=(|A_{r}|-1)\frac{(|A_{1}|+\ldots+|A_{r-1}|)!}{|A_{1}|!\ldots|A_{r-1}|!}\binom{|V(H)|}{|A_{1}|+\ldots+|A_{r-1}|}.
\end{align*}

\noindent\textbf{Second Way:} To simplify the subsequent calculations, we introduce the following notation. We denote $\vec{s}=(s_{1},\ldots,s_{\alpha})\in V(H)^{\alpha}$. When all the coordinates of $\vec{s}$ are distinct, we occasionally write $\vec{\vec{s}}$ in place of $\vec{s}$. Given a subset $S \subseteq V(H)$ with $|S|=\alpha$, we represent $S$ by an ordered tuple $\vec{\vec{s}} \in V(H)^{\alpha}$, obtained by fixing an arbitrary ordering of its elements. With this notation in place, the preceding expression can be rewritten as
\begin{equation*}
 \mathds{1}_{E(H)}(x,S_{1},\ldots,S_{r-1})=\mathds{1}_{E(H)}(x,\vec{\vec{s}}_{1},\ldots,\vec{\vec{s}}_{r-1}).   
\end{equation*}

For each $t\in[r-1]$, let $\vec{s}_{t}=(s_{t1},\ldots,s_{t|A_{t}|})\in V(H)^{|A_{t}|}$. For each $t\in[r-1]$ and $i,j\in[|A_{t}|]$, we construct the following set of $(|A_{1}|+\cdots+|A_{r-1}|)-$tuples: 
\begin{equation*}
 P_{ij}^{(t)}=\left\{(\vec{s}_{1},\ldots,\vec{s}_{t},\ldots,\vec{s}_{r-1})\in V(H)^{|A_{1}|}\times\cdots\times V(H)^{|A_{r-1}|}: s_{ti}=s_{tj}\right\}.  
\end{equation*}
Further, let $T$ denote the set of $(|A_{1}|+\ldots+|A_{r-1}|)-$tuples whose coordinates are pairwise distinct, that is,
\begin{equation*}
T=V(H)^{|A_{1}|}\times\ldots\times V(H)^{|A_{r-1}|}\smallsetminus\left[\underset{t\in[r-1]}{\bigcup}\left(\underset{\substack{i,j\in[|A_{t}|]\\i<j}}{\bigcup}P_{ij}^{(t)}\right)\right].
\end{equation*}

Applying Bonferroni's inequality (Theorem~\ref{Bonferroni's Inequality}) with the odd integer $k=1$ to the last inequality, we obtain
\begin{align*}
|\mathfrak{S}|=&\underset{x\in V(H)}{\sum}\ \underset{\substack{S_{1},\ldots,S_{r-1}\subseteq V(H)\\|s_{1}|=|A_{1}|,\ldots,|S_{r-1}|=|A_{r-1}|}}{\sum}\mathds{1}_{E(H)}(\{x,S_{1},\ldots,S_{r-1}\})\\
=&\underset{x\in V(H)}{\sum}\ \underset{\vec{\vec{s_{1}}}\in V(H)^{|A_{1}|},\ldots,\vec{\vec{s}}_{r-1}\in V(H)^{|A_{r-1}|}}{\sum}\mathds{1}_{E(H)}(\{x,\vec{\vec{s_{1}}},\ldots,\vec{\vec{s}}_{r-1}\})\\
=&\underset{x\in V(H)}{\sum}\ \underset{(\vec{s_{1}},\ldots,\vec{s}_{r-1})\in T}{\sum}\mathds{1}_{E(H)}(\{x,\vec{s_{1}},\ldots,\vec{s}_{r-1}\})\\
=&\underset{x\in V(H)}{\sum}\ \underset{(\vec{s_{1}},\ldots,\vec{s}_{r-1})\in V(H)^{|A_{1}|}\times\ldots\times V(H)^{|A_{r-1}|}}{\sum}\mathds{1}_{E(H)}(\{x,\vec{s_{1}},\ldots,\vec{s}_{r-1}\})\\
&\hspace{5em}-\underset{x\in V(H)}{\sum}\ \underset{(\vec{s_{1}},\ldots,\vec{s}_{r-1})\in \underset{k\in[r-1]}{\bigcup}\left(\underset{\substack{i,j\in[|A_{k}|]\\i<j}}{\bigcup}P_{ij}^{(k)}\right)}{\sum}\mathds{1}_{E(H)}(\{x,\vec{s_{1}},\ldots,\vec{s}_{r-1}\})\\
&\geq\underset{x\in V(H)}{\sum}\ \underset{(\vec{s_{1}},\ldots,\vec{s}_{r-1})\in V(H)^{|A_{1}|}\times\ldots\times V(H)^{|A_{r-1}|}}{\sum}\mathds{1}_{E(H)}(\{x,\vec{s_{1}},\ldots,\vec{s}_{r-1}\})\\
&\hspace{5em}-\underset{x\in V(H)}{\sum}\underset{k\in[r-1]}{\sum}\underset{(\vec{s_{1}},\ldots,\vec{s}_{r-1})\in \underset{\substack{i,j\in[|A_{k}|]\\i<j}}{\bigcup}P_{ij}^{(k)}}{\sum}\mathds{1}_{E(H)}(\{x,\vec{s_{1}},\ldots,\vec{s}_{r-1}\})\\
&=S_{1}-S_{2}\textup{ (say) }.
\end{align*}
Recall that we are estimating $|\mathfrak{S}|$. In the first way we derive an upper bound, whereas in the second way, we need a lower bound. Consequently, it suffices to obtain a lower bound for $S_{1}-S_{2}$.

\begin{example}
Before estimating $S_{1}$ and $S_{2}$ for general values of $r$, we illustrate the computation in the case $r=3$ to facilitate the exposition. In the following estimation of $S_{1}$ for the case $r=3$, we have used the H\"{o}lder's inequality (Theorem~\ref{Holder's Inequality}) multiple (two) times. 
\begin{align*}
S_{1}=&\underset{x\in V(H)}{\sum}\underset{(\vec{u},\vec{v})\in V(H)^{|A_{1}|}\times V(H)^{|A_{2}|}}{\sum}\mathds{1}_{E(H)}(x,\vec{u},\vec{v})\\
=&\underset{x\in V(H)}{\sum}\ \ \underset{\vec{u}\in V(H)^{|A_{1}|}}{\sum}\left(\underset{v\in V(H)}{\sum}\ \ \mathds{1}_{E(H)}(x,\vec{u},v)\right)^{|A_{2}|}\\
\geq&\frac{1}{|V(H)|^{(1+|A_{1}|)(|A_{2}|-1)}}\left[\underset{x\in V(H)}{\sum}\underset{\vec{u}\in V(H)^{|A_{1}|}}{\sum}\underset{v\in V(H)}{\sum}\mathds{1}_{E(H)}(x,\vec{u},v)\right]^{|A_{2}|}\\
=&\frac{1}{|V(H)|^{(1+|A_{1}|)(|A_{2}|-1)}}\left[\underset{x\in V(H)}{\sum}\underset{v\in V(H)}{\sum}\underset{\vec{u}\in V(H)^{|A_{1}|}}{\sum}\mathds{1}_{E(H)}(x,\vec{u},v)\right]^{|A_{2}|}\\
=&\frac{1}{|V(H)|^{(1+|A_{1}|)(|A_{2}|-1)}}\left[\underset{x\in V(H)}{\sum}\underset{v\in V(H)}{\sum}\bigg(\underset{u\in V(H)}{\sum}\mathds{1}_{E(H)}(\{x,u,v\})\bigg)^{|A_{1}|}\right]^{|A_{2}|}\\
\geq&\frac{1}{|V(H)|^{(1+|A_{1}|)(|A_{2}|-1)}}\left[\frac{1}{|V(H)^2|^{|A_{1}|-1}}\bigg(\underset{x\in V(H)}{\sum}\underset{v\in V(H)}{\sum}\underset{u\in V(H)}{\sum}\mathds{1}_{E(H)}(\{x,u,v\})\bigg)^{|A_{1}|}\right]^{|A_{2}|}\\
=&\frac{1}{|V(H)|^{(1+|A_{1}|)(|A_{2}|-1)}}\frac{1}{|V(H)|^{2(|A_{1}|-1)|A_{2}|}}\bigg(\underset{x\in V(H)}{\sum}\underset{u\in V(H)}{\sum}\underset{v\in V(H)}{\sum}\mathds{1}_{E(H)}(\{x,u,v\})\bigg)^{|A_{1}||A_{2}|}\\
=&\frac{(3|E(H)|)^{|A_{1}||A_{2}|}}{|V(H)|^{3|A_{1}||A_{2}|-|A_{1}|-|A_{2}|-1}}
\end{align*}
For the illustrative case $r=3$, we estimate the second term $S_{2}$. Here using principle of inclusion and exclusion over the index set we get,
\begin{align*}
S_{2}&=\underset{x\in V(H)}{\sum}\ \ \underset{(\vec{u},\vec{v})\in \underset{\substack{i,j\in[|A_{1}|]\\ i\neq j}}{\bigcup}P^{(1)}_{ij}}{\sum}\ \ \mathds{1}_{E(H)}(x,\vec{u},\vec{v})+\\
&\hspace{1cm}+\underset{x\in V(H)}{\sum}\ \ \underset{(\vec{u},\vec{v})\in \underset{\substack{k,l\in[|A_{2}|]\\ k\neq l}}{\bigcup}P^{(2)}_{kl}}{\sum}\ \ \mathds{1}_{E(H)}(x,\vec{u},\vec{v})-\\
&\hspace{1cm}-\underset{x\in V(H)}{\sum}\ \ \underset{(\vec{u},\vec{v})\in \bigg(\underset{\substack{i,j\in[|A_{1}|]\\ i\neq j}}{\bigcup}P^{(1)}_{ij}\bigg)\cap\bigg(\underset{\substack{k,l\in[|A_{2}|]\\ k\neq l}}{\bigcup}P^{(2)}_{kl}\bigg)}{\sum}\ \ \mathds{1}_{E(H)}(x,\vec{u},\vec{v})\\ 
&= S_{21}+S_{22}-S_{23}\text{ (say). }
\end{align*}
Applying Bonferroni's inequality (Theorem~\ref{Bonferroni's Inequality}) with the odd integer $k=1$, we have 
\begin{align*}
S_{21}=&\underset{x\in V(H)}{\sum}\ \ \underset{(\vec{u},\vec{v})\in \underset{\substack{i,j\in[|A_{1}|]\\ i\ne j}}{\bigcup}P^{(1)}_{ij}}{\sum}\ \ \mathds{1}_{E(H)}(x,\vec{u},\vec{v})\\
\leq&\underset{x\in V(H)}{\sum}\ \ \underset{\{i,j\}\subseteq\binom{[|A_{1}|]}{2}}{\sum}\underset{(\vec{u},\vec{v})\in P^{(1)}_{ij}}{\sum}\ \ \mathds{1}_{E(H)}(x,\vec{u},\vec{v})\\
=&\underset{\{i,j\}\subseteq\binom{[|A_{1}|]}{2}}{\sum}\underset{x\in V(H)}{\sum}\underset{(\vec{u},\vec{v})\in P^{(1)}_{ij}}{\sum}\ \ \mathds{1}_{E(H)}(x,\vec{u},\vec{v})\\
\end{align*}
Let $\{i,j\}\in \binom{[|A_{1}|]}{2}$. Then  between $P^{(1)}_{ij}$ and $V(H)^{|A_{1}|-1}\times V(H)^{|A_{2}|}$, there exists an one-to-one and onto map, which is the \emph{projection map} defined as 
\begin{equation*}
 (\vec{u},\vec{v})\mapsto((u_{1},...,u_{i-1},u_{i},u_{i+1},\ldots,u_{j-1},u_{j+1},...,u_{|A|}),\vec{v}).  
\end{equation*}
Hence, the above sum becomes,
\begin{align*}
S_{21}=&\underset{\{i,j\}\subseteq\binom{[|A_{1}|]}{2}}{\sum}\underset{x\in V(H)}{\sum}\underset{(\vec{u},\vec{v})\in V(H)^{|A_{1}|-1}\times V(H)^{|A_{2}|}}{\sum}\ \ \mathds{1}_{E(H)}(x,\vec{u},\vec{v})\\
=&\binom{|A_{1}|}{2}\underset{x\in V(H)}{\sum}\underset{(\vec{u},\vec{v})\in V(H)^{|A_{1}|-1}\times V(H)^{|A_{2}|}}{\sum}\ \ \mathds{1}_{E(H)}(x,\vec{u},\vec{v})\\
=&\binom{|A_{1}|}{2}\underset{x\in V(H)}{\sum}\underset{\vec{u}\in V(H)^{|A_{1}|-1}}{\sum}\left(\underset{v\in V(H)}{\sum}\mathds{1}_{E(H)}(x,\vec{u},v)\right)^{|A_{2}|}\\
\leq&\binom{|A_{1}|}{2}\underset{x\in V(H)}{\sum}\underset{\vec{u}\in V(H)^{|A_{1}|-1}}{\sum}|V(H)|^{|A_{2}|-1}\underset{v\in V(H)}{\sum}\mathds{1}_{E(H)}(x,\vec{u},v)\\
=&\binom{|A_{1}|}{2}|V(H)|^{|A_{2}|-1}\underset{x,v\in V(H)}{\sum}\underset{\vec{u}\in V(H)^{|A_{1}|-1}}{\sum}\mathds{1}_{E(H)}(x,\vec{u},v)\\
=&\binom{|A_{1}|}{2}|V(H)|^{|A_{2}|-1}\underset{x,v\in V(H)}{\sum}\left(\underset{u\in V(H)}{\sum}\mathds{1}_{E(H)}(\{x,u,v\})\right)^{|A_{1}|-1}\\
\leq&\binom{|A_{1}|}{2}|V(H)|^{|A_{2}|-1}\underset{x,v\in V(H)}{\sum}|V(H)|^{|A_{1}|-2}\underset{u\in V(H)}{\sum}\mathds{1}_{E(H)}(\{x,u,v\})\\
=&\binom{|A_{1}|}{2}|V(H)|^{|A_{1}|+|A_{2}|-3}\underset{x,u,v\in V(H)}{\sum}\mathds{1}_{E(H)}(\{x,u,v\})\\
=&\binom{|A_{1}|}{2}|V(H)|^{|A_{1}|+|A_{2}|-3}3|E(H)|.
\end{align*}

A similar calculation shows that
\begin{equation*}
S_{22}\leq\binom{|A_{2}|}{2}|V(H)|^{|A_{1}|+|A_{2}|-3}3|E(H)|.    
\end{equation*} 
Since $S_{23}\geq0$, we have the upper bound  
\begin{align*}
S_{2}=S_{21}+S_{22}-S_{23}\leq S_{21}+S_{22}\leq\bigg[\binom{|A_{1}|}{2}+\binom{|A_{2}|}{2}\bigg]|V(H)|^{|A_{1}|+|A_{2}|-3}3|E(H)|.   
\end{align*}
Consequently, for the case $r=3$, we obtain the following lower bound for $S_{1}-S_{2}$, that is
\begin{align*}
 S_{1}-S_{2}&\geq S_{1}-S_{21}-S_{22}\\
 &\geq\frac{(3|E(H)|)^{|A_{1}||A_{2}|}}{|V(H)|^{3|A_{1}||A_{2}|-|A_{1}|-|A_{2}|-1}}-\bigg[\binom{|A_{1}|}{2}+\binom{|A_{2}|}{2}\bigg]|V(H)|^{|A_{1}|+|A_{2}|-3}3|E(H)|.  
\end{align*}
\end{example}

We now proceed to estimate $S_{1}$ and $S_{2}$ for general $r$. As in the case $r=3$, the estimation of $S_{1}$ is obtained by applying H\"{o}lder's inequality (Theorem~\ref{Holder's Inequality}) repeatedly, a total of $(r-1)$ times, to derive the desired expression. Finally, we apply the Lemma~\ref{exponent} to simplify the exponent of $|V(H)|$.
\begin{align*}
S_{1}=&\underset{x\in V(H)}{\sum}\underset{\substack{\vec{s}_{1}\in V(H)^{|A_{1}|}\\\vdots\\\vec{s}_{r-1}\in V(H)^{|A_{r-1}|}}}{\sum}\mathds{1}_{E(H)}(x,\vec{s_{1}},\ldots,\vec{s}_{r-1})\\
=&\underset{x\in V(H)}{\sum}\underset{\vec{s_{1}}\in V(H)^{|A_{1}|}}{\sum}\ldots\underset{\vec{s}_{r-2}\in V(H)^{|A_{r-2}|}}{\sum}\left(\underset{s_{r-1}\in V(H)}{\sum}\mathds{1}_{E(H)}(x,\vec{s_{1}},\ldots,\vec{s}_{r-2},s_{r-1})\right)^{|A_{r-1}|}\\
\geq&\frac{1}{|V(H)|^{(1+|A_{1}|+\ldots+|A_{r-2}|)(|A_{r-1}|-1)}}\\
&\hspace{1cm}\Bigg[\underset{x\in V(H)}{\sum}\underset{\vec{s_{1}}\in V(H)^{|A_{1}|}}{\sum}\ldots\underset{\vec{s}_{r-2}\in V(H)^{|A_{r-2}|}}{\sum}\underset{s_{r-1}\in V(H)}{\sum}\mathds{1}_{E(H)}(x,\vec{s_{1}},\ldots,\vec{s}_{r-2},s_{r-1})\Bigg]^{|A_{r-1}|}\\
=&\frac{1}{|V(H)|^{(1+|A_{1}|+\ldots+|A_{r-2}|)(|A_{r-1}|-1)}}\\
&\hspace{1cm}\Bigg[\underset{x,s_{r-1}\in V(H)}{\sum}\underset{\vec{s_{1}}\in V(H)^{|A_{1}|}}{\sum}\ldots\underset{\vec{s}_{r-2}\in V(H)^{|A_{r-2}|}}{\sum}\mathds{1}_{E(H)}(x,\vec{s_{1}},\ldots,\vec{s}_{r-2},s_{r-1})\Bigg]^{|A_{r-1}|}\\
=&\frac{1}{|V(H)|^{(1+|A_{1}|+\ldots+|A_{r-2}|)(|A_{r-1}|-1)}}\Bigg[\underset{x,s_{r-1}\in V(H)}{\sum}\underset{\vec{s_{1}}\in V(H)^{|A_{1}|}}{\sum}\ldots\underset{\vec{s}_{r-3}\in V(H)^{|A_{r-3}|}}{\sum}\\
&\hspace{1cm}\left(\underset{s_{r-2}\in V(n^H)}{\sum}\mathds{1}_{E(H)}(\{x,\vec{s_{1}},\ldots,\vec{s}_{r-3},s_{r-2},s_{r-1}\})\right)^{|A_{r-2}|}\Bigg]^{|A_{r-1}|}\\
\geq&\frac{1}{|V(H)|^{(1+|A_{1}|+\ldots+|A_{r-2}|)(|A_{r-1}|-1)}}\Bigg[\frac{1}{|V(H)|^{(2+|A_{1}|+\ldots+|A_{r-3}|)(|A_{r-2}|-1)}}\Bigg(\underset{x,s_{r-1}\in V(H)}{\sum}\\
&\hspace{1cm}\ldots\underset{\vec{s_{1}}\in V(H)^{|A_{1}|}}{\sum}\underset{\vec{s}_{r-3}\in V(H)^{|A_{r-3}|}}{\sum}\underset{s_{r-2}\in V(H)}{\sum}\mathds{1}_{E(H)}(x,\vec{s_{1}},\ldots,\vec{s}_{r-3},s_{r-2},s_{r-1})\Bigg)^{|A_{r-2}|}\Bigg]^{|A_{r-1}|}\\
=&\frac{1}{|V(H)|^{(1+|A_{1}|+\ldots+|A_{r-2}|)(|A_{r-1}|-1)+(2+|A_{1}|+\ldots+|A_{r-3}|)(|A_{r-2}|-1)|A_{r-1}|}}\Bigg[\underset{x,s_{r-1}\in V(H)}{\sum}\underset{\vec{s_{1}}\in V(H)^{|A_{1}|}}{\sum}\\
&\hspace{1cm}\ldots\underset{\vec{s}_{r-3}\in V(H)^{|A_{r-3}|}}{\sum}\underset{s_{r-2}\in V(H)}{\sum}\mathds{1}_{E(H)}(x,\vec{s_{1}},\ldots,\vec{s}_{r-3},s_{r-2},s_{r-1})\Bigg]^{|A_{r-2}||A_{r-1}|}\\
&\hspace{1cm}\vdots\\
=&\frac{1}{|V(H)|^{(1+|A_{1}|+\ldots+|A_{r-2}|)(|A_{r-1}|-1)+(2+|A_{1}|+\ldots+|A_{r-3}|)(|A_{r-2}|-1)|A_{r-1}|+\ldots+(r-1)(|A_{1}|-1)|A_{2}|\ldots|A_{r-1}|}}\\
&\hspace{7em}\Bigg[\underset{x,s_{1},\ldots,s_{r-1}\in V(H)}{\sum}\mathds{1}_{E(H)}(\{x,s_{1},\ldots,s_{r-1}\})\Bigg]^{|A_{1}|\ldots|A_{r-1}|}\\
=&\frac{(r|E(H)|)^{|A_{1}|\ldots|A_{r-1}|}}{|V(H)|^{r|A_{1}|\ldots|A_{r-1}|-(|A_{1}|+\ldots+|A_{r-1}|)-1}}.
\end{align*}

We proceed similarly to the previous example of case $r=3$ to derive an upper bound for
\begin{align*}
S_{2}=&\underset{x\in V(H)}{\sum}\underset{k\in[r-1]}{\sum}\underset{(\vec{s_{1}},\ldots,\vec{s}_{r-1})\in \underset{\substack{i,j\in[|A_{k}|]\\i<j}}{\bigcup}P_{ij}^{(k)}}{\sum}\mathds{1}_{E(H)}(x,\vec{s_{1}},\ldots,\vec{s}_{r-1})\\
=&\underset{k\in[r-1]}{\sum}\underset{x\in V(H)}{\sum}\underset{(\vec{s_{1}},\ldots,\vec{s}_{r-1})\in \underset{\substack{i,j\in[|A_{k}|]\\i<j}}{\bigcup}P_{ij}^{(k)}}{\sum}\mathds{1}_{E(H)}(x,\vec{s_{1}},\ldots,\vec{s}_{r-1})\\
=&S_{21}+\ldots+S_{2(r-1)}\text{ (say), }
\end{align*}
where for each $t\in[r-1]$
\begin{equation*}
S_{2t}=\underset{x\in V(H)}{\sum}\underset{(\vec{s_{1}},\ldots,\vec{s}_{r-1})\in \underset{\substack{i,j\in[|A_{t}|]\\i<j}}{\bigcup}P_{ij}^{(t)}}{\sum}\mathds{1}_{E(H)}(x,\vec{s_{1}},\ldots,\vec{s}_{r-1}).  
\end{equation*}

Again using Bonferroni's inequality (Theorem~\ref{Bonferroni's Inequality}) with the odd integer $k=1$, we have 
\begin{align*}
S_{21}=&\underset{x\in V(H)}{\sum}\underset{(\vec{s_{1}},\ldots,\vec{s}_{r-1})\in \underset{\substack{i,j\in[|A_{1}|]\\i<j}}{\bigcup}P_{ij}^{(1)}}{\sum}\mathds{1}_{E(H)}(x,\vec{s_{1}},\ldots,\vec{s}_{r-1})\\
\leq&\underset{1\leq i<j\leq|A_{1}|}{\sum}\underset{x\in V(H)}{\sum}\underset{(\vec{s_{1}},\ldots,\vec{s}_{r-1})\in P_{ij}^{(1)}}{\sum}\mathds{1}_{E(H)}(x,\vec{s_{1}},\ldots,\vec{s}_{r-1})\\
=&\underset{\{i,j\}\in\binom{[|A_{1}|]}{2}}{\sum}\underset{x\in V(H)}{\sum}\underset{\substack{\vec{s}_{1}\in V(H)^{|A_{1}|-1}\\\vec{s}_{2}\in V(H)^{|A_{2}|}\\\vdots\\\vec{s}_{r-1}\in V(H)^{|A_{r-1}|}}}{\sum}\mathds{1}_{E(H)}(x,\vec{s}_{1},\vec{s}_{2},\ldots,\vec{s}_{r-1}).
\end{align*}
The last expression holds, since there is an one-to-one and onto map between $P^{(1)}_{ij}$ and 
\begin{equation*}
 V(H)^{|A_{1}|-1}\times V(H)^{|A_{2}|}\times\cdots\times V(H)^{|A_{r-1}|}.   
\end{equation*}
We continue calculating the above extremal term in the following way.
\begin{align*}
S_{21}\leq&\binom{|A_{1}|}{2}\underset{x\in V(H)}{\sum}\underset{\substack{\vec{s}_{1}\in V(H)^{|A_{1}|-1}\\\vec{s}_{2}\in V(H)^{|A_{2}|}\\\vdots\\\vec{s}_{r-1}\in V(H)^{|A_{r-1}|}}}{\sum}\mathds{1}_{E(H)}(x,\vec{s}_{1},\vec{s}_{2}\ldots,\vec{s}_{r-1})\\
=&\binom{|A_{1}|}{2}\underset{x\in V(H)}{\sum}\underset{\substack{\vec{s}_{2}\in V(H)^{|A_{2}|}\\\vdots\\\vec{s}_{r-1}\in V(H)^{|A_{r-1}|}}}{\sum}\left(\underset{s_{1}\in V(H)}{\sum}\mathds{1}_{E(H)}(x,s_{1},\vec{s}_{2}\ldots,\vec{s}_{r-1})\right)^{|A_{1}|-1}\\
\leq&\binom{|A_{1}|}{2}\underset{x\in V(H)}{\sum}\underset{\substack{\vec{s}_{2}\in V(H)^{|A_{2}|}\\\vdots\\\vec{s}_{r-1}\in V(H)^{|A_{r-1}|}}}{\sum}|V(H)|^{|A_{1}|-2}\underset{s_{1}\in V(H)}{\sum}\mathds{1}_{E(H)}(x,s_{1},\vec{s}_{2}\ldots,\vec{s}_{r-1})\\
\end{align*}
The last inequality holds, since $\mathds{1}_{E(H)}(x,s_{1},\vec{s}_{2}\ldots,\vec{s}_{r-1})\leq1$. Continuing this way, the above sum is as follows:-
\begin{align*}
S_{21}\leq&\binom{|A_{1}|}{2}\underset{x\in V(H)}{\sum}|V(H)|^{(|A_{1}|-2)+(|A_{2}|-1)+\ldots+(|A_{r-1}|-1)}\underset{s_{1},\ldots,s_{r-1}\in V(H)}{\sum}\mathds{1}_{E(H)}(\{x,s_{1},\ldots,s_{r-1}\})\\
=&\binom{|A_{1}|}{2}|V(H)|^{|A_{1}|+\ldots+|A_{r-1}|-r}\underset{x,s_{1},\ldots,s_{r-1}\in V(H)}{\sum}\mathds{1}_{E(H)}(\{x,s_{1},\ldots,s_{r-1}\})\\
=&\binom{|A_{1}|}{2}|V(H)|^{|A_{1}|+\ldots+|A_{r-1}|-r}r|E(H)|.
\end{align*}

Similarly, for each $t\in[r-1]$,
\begin{align*}
S_{2t}\leq \binom{|A_{t}|}{2}|V(H)|^{|A_{1}|+\ldots+|A_{r-1}|-r}r|E(H)|.
\end{align*}

Hence, 
\begin{align*}
S_{2}&\leq\underset{t\in[r-1]}{\sum}\binom{|A_{t}|}{2}|V(H)|^{|A_{1}|+|A_{2}|+\ldots+|A_{r-1}|-r}\ r|E(H)|\\
&\leq (r-1)\binom{|A_{r-1}|}{2}|V(H)|^{|A_{1}|+|A_{2}|+\ldots+|A_{r-1}|-r}r|E(H)|,
\end{align*}
The last inequality holds, since we have assumed $|A_{1}|\leq\ldots\leq|A_{r-1}|$. 

Therefore,
\begin{tcolorbox}
\begin{align}\label{inequality r-graph}
&\frac{(r|E(H)|)^{|A_{1}|\ldots|A_{r-1}|}}{|V(H)|^{r|A_{1}|\ldots|A_{r-1}|-(|A_{1}|+\ldots+|A_{r-1}|)-1}}-r(r-1)\binom{|A_{r-1}|}{2}|V(H)|^{|A_{1}|+|A_{2}|+\ldots+|A_{r-1}|-r}|E(H)|\nonumber\\
&\hspace{1cm}\leq|\mathfrak{S}|\leq(|A_{r}|-1)\frac{(|A_{1}|+\ldots+|A_{r-1}|)!}{|A_{1}|!\ldots|A_{r-1}|!}\binom{|V(H)|}{|A_{1}|+\ldots+|A_{r-1}|}\tag{$\star$}.
\end{align}
\end{tcolorbox}

\subsection{Establishment of the upper bound of \texorpdfstring{$|E(H)|$}{|E(H)|} and conclusion}

Using the above estimates of $|\mathfrak{S}|$, we have 
\begin{align*}
&\frac{(r|E(H)|)^{|A_{1}|\ldots|A_{r-1}|}}{|V(H)|^{r|A_{1}|\ldots|A_{r-1}|-(|A_{1}|+\ldots+|A_{r-1}|)-1}}-r(r-1)\binom{|A_{r-1}|}{2}|V(H)|^{|A_{1}|+|A_{2}|+\ldots+|A_{r-1}|-r}|E(H)|\\
&\hspace{5cm}\leq(|A_{r}|-1)\frac{1}{|A_{1}|!\ldots|A_{r-1}|!}\frac{|V(H)|!}{\left(|V(H)|-|A_{1}|\ldots-|A_{r-1}|\right)!}.
\end{align*}
After further simplification, this becomes:
\begin{align*}
&\frac{r|E(H)|}{|V(H)|^{r-\frac{|A_{1}|+\ldots+|A_{r-1}|}{|A_{1}|\ldots|A_{r-1}|}-\frac{1}{|A_{1}|\ldots|A_{r-1}|}}}\\
&\leq\Bigg[(|A_{r}|-1)\frac{1}{|A_{1}|!\ldots|A_{r-1}|!}\frac{|V(H)|!}{\left(|V(H)|-|A_{1}|\ldots-|A_{r-1}|\right)!}+\\
&\hspace{3cm}+r(r-1)\binom{|A_{r-1}|}{2}|V(H)|^{|A_{1}|+|A_{2}|+\ldots+|A_{r-1}|-r}|E(H)|\Bigg]^{\frac{1}{|A_1|\ldots|A_{r-1}|}}\\
&\leq\Bigg[(|A_{r}|-1)\frac{|V(H)|^{|A_{1}|+\ldots+|A_{r-1}|}}{|A_{1}|!\ldots|A_{r-1}|!}+\\
&\hspace{3cm}+r(r-1)\binom{|A_{r-1}|}{2}|V(H)|^{|A_{1}|+|A_{2}|+\ldots+|A_{r-1}|-r}\frac{|V(H)|^r}{r!}\Bigg]^{\frac{1}{|A_1|\ldots|A_{r-1}|}}\\
&\hspace{2cm}=|V(H)|^{\frac{|A_{1}|+\ldots+|A_{r-1}|}{|A_{1}|\ldots|A_{r-1}|}}\Bigg(\frac{|A_{r}|-1}{|A_{1}|!\ldots|A_{r-1}|!}+\frac{r(r-1)}{r!}\binom{|A_{r-1}|}{2}\Bigg)^{\frac{1}{|A_1|\ldots|A_{r-1}|}}.
\end{align*}
In the above, we have used the standard inequality for binomial coefficients, namely $\binom{N}{K}\leq\frac{N^{K}}{K!}$, where $N$ and $K$ are positive integers. The said expression simplifies to 
\begin{equation*}
|E(H)|\leq|V(H)|^{r-\frac{1}{|A_{1}|\ldots|A_{r-1}|}}\frac{1}{r}\Bigg(\frac{|A_{r}|-1}{|A_{1}|!\ldots|A_{r-1}|!}+\frac{r(r-1)}{r!}\binom{|A_{r-1}|}{2}\Bigg)^{\frac{1}{|A_1|\ldots|A_{r-1}|}}.
\end{equation*}
Since $H$ has been chosen arbitrarily, we have for each positive integer $n$, 
\begin{equation*}
\mathrm{ex}(n,\mathbb{K}^{(r)}[A_{1},\ldots,A_{r}])\leq Cn^{r-\frac{1}{|A_{1}|\ldots|A_{r-1}|}},
\end{equation*}
where
\begin{equation*}
C=\frac{1}{r}\Bigg(\frac{|A_{r}|-1}{|A_{1}|!\ldots|A_{r-1}|!}+\frac{r(r-1)}{r!}\binom{|A_{r-1}|}{2}\Bigg)^{\frac{1}{|A_1|\ldots|A_{r-1}|}}.    
\end{equation*}
This shows that $C$ depends on $\{r,|A_{1}|,\ldots,|A_{r-1}|,|A_{r}|\}$ and completes the proof of Theorem~\ref{upper bound of Erdos Result}.

\noindent\textbf{Methodological Summary.} The proof begins by expressing the counting problem as an appropriate sum of indicator functions. This converts the combinatorial problem into an analytic one. Repeated applications of H\"{o}lder’s inequality produce recursive estimates. The multiple sums are then rearranged using a discrete analogue of Fubini’s theorem. Finally, Bonferroni’s inequality controls the repeated configurations. Together, these ingredients yield the desired extremal estimate. The preceding argument illustrates a general analytic counting framework based on $\mathsf{L}^{p}$-techniques. For convenience, we refer to this framework as the \emph{$\mathsf{L}^{p}-$method}. We believe that this framework may be applicable to other extremal problems involving multilinear counting expressions. 

\begin{acknowledgement}
The author Subhankar Dash acknowledges the National Institute of Science Education and Research (NISER), Bhubaneswar and the Homi Bhabha National Institute (HBNI), Mumbai, for financial support through a doctoral fellowship.
\end{acknowledgement}

\bibliographystyle{amsplain}

\end{document}